\documentclass{article}

\usepackage{arxiv}

\usepackage[utf8]{inputenc} 
\usepackage[T1]{fontenc}    
\usepackage{hyperref}       
\usepackage{url}            
\usepackage{booktabs}       
\usepackage{amsfonts}       
\usepackage{nicefrac}       
\usepackage{microtype}      
\usepackage{lipsum}		

\usepackage{mathptmx, amsmath, amssymb, amsfonts, amsthm, mathptmx, enumerate, color}

\numberwithin{equation}{section}
\newtheorem{theorem}{Theorem}[section]
\newtheorem{lemma}[theorem]{Lemma}

\newtheorem{definition}[theorem]{Definition}

\title{Theory of Nonlinear Caputo-Katugampola Fractional Differential Equations}

\author{
  Saleh S. Redhwan \thanks{ S.S. Redhwan, saleh.redhwan909@gmail.com.} \\
 Department of Mathematics\\
 Dr.Babasaheb Ambedkar Marathwada\\
University, Aurangabad, (M.S),431001, India\\
  \texttt{saleh.redhwan909@gmail.com} \\
   \And
 Sadikali L. Shaikh\\
  Department of Mathematics\\
 Maulana Azad College of arts, \\
 Science and Commerce \\
 Aurangabad, (M.S),431001, India\\
  \texttt{} \\
   \AND
  Mohammed S. Abdo$^1$$^,$$^2$\\
  $^1$Department of Mathematics\\
 Dr.Babasaheb Ambedkar Marathwada  University\\
Aurangabad, (M.S),431001, India\\
 $^2$Hodeidah University\\
 Al-Hodeidah, 3114, Yemen\\
  \texttt{msabdo1977@gmail.com} \\
}

\begin{document}
\maketitle

\begin{abstract}
This manuscript investigates the existence and uniqueness of solutions to
the first order fractional anti-periodic boundary value problem involving
Caputo-Katugampola (CK) derivative. A variety of tools for analysis this
paper through the integral equivalent equation of the given problem, fixed
point theorems of Leray--Schauder, Krasnoselskii's, and Banach are used.
Examples of the obtained results are also presented.
\end{abstract}


\section{Introduction}

In this manuscript, we investigate the sufficient conditions of the
existence and uniqueness of solutions for the CK fractional differential
equation%
\begin{equation}
^{c}D_{a^{+}}^{\alpha ;\rho }y(t)=f(t,y(t)),\text{ where }0<\alpha <1,\text{
}t\in J=[a,b].  \label{1}
\end{equation}%
with the anti-period boundary condition
\begin{equation}
y(a)+y(b)=0,  \label{2}
\end{equation}%
where $^{c}D_{a^{+}}^{\alpha ;\rho }$ is CK fractional derivatives of order $%
\alpha $, and $f:J\times
\mathbb{R}
\longrightarrow
\mathbb{R}
$ is the appropriate function.

The subject of fractional differential equations has recently evolved as an
interesting and popular field of research; see the interesting books \cite{k}
\cite{p}. In fact, fractional derivatives provide an excellent tool for the
description of memory and hereditary properties of various materials and
processes. More and more researchers have found that fractional differential
equations play important roles in many research areas, such as physics,
chemical technology, population dynamics, biotechnology, and economics. For
examples and recent developments of the topic, we refer to the papers \cite%
{AZ} \cite{BK} \cite{BO} and the references cited therein.

In \cite{K1} the author introduced a new fractional integral, which
generalizes the Riemann-Liouville and the Hadamard integral into a single
form. For further properties such as expansion formulas, variational
calculus applications, control theoretical applications, convexity and
integral inequalities and Hermite-Hadamard type inequalities of this new
operator and similar operators, for example, see \cite{AG} \cite{SG} \cite%
{RH} \cite{GJ}. The corresponding fractional derivatives were introduced in
\cite{K2} \cite{K3} \cite{MO}, Which so-called katugampola fractional
operators. The existence and uniqueness results for the CK derivative are
given in \cite{K4}, the author used the Peano theorem to obtain the
existence and uniqueness results of the Cauchy type problem%
\begin{equation}
^{c}D_{0^{+}}^{\alpha ;\rho }y(t)=f(t,y(t)),\text{ \ \ }t\in \lbrack 0,b],
\label{3}
\end{equation}%
\begin{equation}
y^{(k)}(0)=y_{0}^{(k)},\text{ \ }k=0,1,...,m-1,  \label{4}
\end{equation}%
where $m=[\alpha ].$ In the same context, R. Almeida in \cite{RA}, proved
the uniqueness result of the problem (\ref{3})-(\ref{4}) involving $%
^{c}D_{a^{+}}^{\alpha ;\rho }$ via Gronwall inequality type.

On the other hand, Oliveira and de Oliveira in \cite{AM}, considered the
initial value problem for a nonlinear fractional differential equation
including Hilfer-Katugampola derivative of the form%
\begin{equation}
^{\rho }D_{a^{+}}^{\alpha ,\beta }y(t)=f(t,y(t)),\text{ \ \ }t\in J,
\label{eq17}
\end{equation}%
\begin{equation}
^{\rho }I_{a^{+}}^{1-\gamma }y(a)=c,\text{ \ }\gamma =\alpha +\beta -\alpha
\beta .  \label{eq18}
\end{equation}%
They used the generalized Banach fixed point theorem to investigate the
existence and uniqueness results on problems (\ref{eq17})-(\ref{eq18}).

The recent development of fractional differential equations and the theoretical analysis can be seen in \cite{AP1,AP2,AP3,AP4,BB,K4}. As such, the CK type fractional operators have not been studied and
investigated in much detail as yet as compared to the other classical
operators. In this manuscript, we introduce new results of the solution to
the problem (\ref{1})-(\ref{2}) involving CK fractional operator. Further,
we use some fixed point theorems to analyize our results. To the best of our
knowledge, the fractional boundary value problems involving
Caputo-Katugampola type fractional operators have not yet been investigated
and developed till the present day.

The paper is organized as follows. In the next section, we begin by
summarizing the forms of Katugampola fractional integral and CK type
fractional derivative, and we also present the background material related
to our work and prove an important lemma which plays a key role in the
sequel. The third section contains the existence and uniqueness results to
the problem (\ref{1})-(\ref{2}) by means of fixed point theorems of (Banach,
Leray-schauder, Krasnoselskii's).\ In the last section, we present some
illustrative examples.

\section{Preliminaries}

In this partition, we recall some essential basic definitions, lemmas, and
preliminary facts related to our results throughout the paper. Let $J=[a,b]$
$\left( -\infty <a<b<\infty \right) \ $be a finite interval of $%
\mathbb{R}
$. Denote $\mathcal{C(}J,%
\mathbb{R}
)$ be the Banach space of all continuous functions $h:J\longrightarrow
\mathbb{R}
$ endowed with the norm given by%
\begin{equation*}
\left\Vert h\right\Vert _{\mathcal{C}}=\sup_{t\in J}\left\vert
h(t)\right\vert :t\in J\},
\end{equation*}%
where $h\in \mathcal{C(}J,%
\mathbb{R}
)$. $\mathcal{C}^{n}\mathcal{(}J,%
\mathbb{R}
)$ $(n\in
\mathbb{N}
_{0})$ denotes the set of mappings having $n$ times continuously
differentiable on $J.$

For $a<b$, $c\in
\mathbb{R}
$ and $1\leq p<\infty ,$ define the function space%
\begin{equation*}
X_{c}^{p}(a,b)=\left\{ h:J\rightarrow
\mathbb{R}
:\left\Vert h\right\Vert _{X_{c}^{p}}=\left( \int_{a}^{b}\left\vert
t^{c}h(t)\right\vert ^{p}\frac{dt}{t}\right) ^{\frac{1}{p}}<\infty \right\} .
\end{equation*}%
for $p=\infty ,$
\begin{equation*}
\left\Vert h\right\Vert _{X_{c}^{p}}=ess\sup_{a\leq t\leq b}\left[
\left\vert t^{c}h(t)\right\vert \right] .
\end{equation*}

Now we state some definitions of the generalized fractional operators were
introduced in \cite{K2,K4,K1}.

\begin{definition}
\label{de1} Let $t>a$ be two reals, $\alpha >0$, $\rho >0,$ and $%
h:J\rightarrow
\mathbb{R}
$ be an integrable function. The left-sided Katugampola fractional integral
of order $\alpha $ and type $\rho $ is defined by
\begin{equation}
I_{a^{+}}^{\alpha ;\rho }h(t)=\frac{1}{\Gamma (\alpha )}\int_{a}^{t}s^{\rho
-1}\left( \frac{t^{\rho }-s^{\rho }}{\rho }\right) ^{\alpha -1}h(s)ds,
\label{a1}
\end{equation}%
where, $\Gamma (.)$ is a gamma function.
\end{definition}

\begin{definition}
Let $n-1<\alpha <n$ , $(n=[\alpha ]+1),$ $\rho >0,$ and $h\in \mathcal{C}^{n}%
\mathcal{(}J,%
\mathbb{R}
).$ The left-sided Katugampola fractional derivative of order $\alpha $ with
dependence on a parameter $\rho $ is defined as%
\begin{equation}
D_{a^{+}}^{\alpha ;\rho }h(t)=\left( t^{1-\rho }\frac{d}{dt}\right)
^{n}I_{a^{+}}^{n-\alpha ;\rho }h(t)=\frac{\gamma ^{n}}{\Gamma (1-\alpha )}%
\int_{a}^{t}\left( \frac{t^{\rho }-s^{\rho }}{\rho }\right) ^{n-\alpha -1}%
\frac{h(s)}{s^{1-\rho }}ds\text{, }t>a,  \label{e2}
\end{equation}%
where $\gamma =\left( t^{1-\rho }\frac{d}{dt}\right) .$ In particular, if $%
0<\alpha <1,$ $\rho >0,$ and $h\in C^{1}(J,%
\mathbb{R}
),$ we have
\begin{equation*}
D_{a^{+}}^{\alpha ;\rho }h(t)=\left( t^{1-\rho }\frac{d}{dt}\right)
I_{a^{+}}^{1-\alpha ;\rho }h(t)=\frac{\gamma }{\Gamma (1-\alpha )}%
\int_{a}^{t}\left( \frac{t^{\rho }-s^{\rho }}{\rho }\right) ^{-\alpha }\frac{%
h(s)}{s^{1-\rho }}ds\text{, }t>a.
\end{equation*}
\end{definition}

\begin{definition}
Let $\alpha \geq 0,$ $n=[\alpha ]+1.$ If $h\in \mathcal{C}^{n}\mathcal{(}J,%
\mathbb{R}
),$ we define the left sided CK fractional derivatives of $h$ of order $%
\alpha $ with a parameter $\rho >0$ by%
\begin{equation}
^{C}D_{a^{+}}^{\alpha ;\rho }h(t)=D_{a^{+}}^{\alpha ;\rho }\left[
h(t)-\sum_{k=0}^{n-1}\frac{h_{\rho }^{(k)}(a)}{k!}\left( \frac{t^{\rho
}-a^{\rho }}{\rho }\right) ^{k}\right] ,  \label{e3}
\end{equation}%
where $h_{\rho }^{(k)}(t)=\left( t^{1-\rho }\frac{d}{dt}\right) ^{k}h(t).$
In case $0<\alpha <1$, and $h\in \mathcal{C}^{1}\mathcal{(}J,%
\mathbb{R}
)$ we have
\begin{equation*}
^{C}D_{a^{+}}^{\alpha ;\rho }h(t)=D_{a^{+}}^{\alpha ;\rho }\left[ h(t)-h(a)%
\right] .
\end{equation*}
From last equation and (\ref{e2}), one deduces
\begin{equation*}
^{C}D_{a^{+}}^{\alpha ;\rho }h(t)=\frac{\gamma }{\Gamma (1-\alpha )}%
\int_{a}^{t}\left( \frac{t^{\rho }-s^{\rho }}{\rho }\right) ^{-\alpha }\frac{%
\left[ h(s)-h(a)\right] }{s^{1-\rho }}ds\text{, }t>a.
\end{equation*}%
Notice that, if $\alpha \notin
\mathbb{N}
_{0}$, $h$ is an absolutely continuous function on $J,$ then the CK
fractional derivative exists a.e. Moreover, we have%
\begin{eqnarray*}
^{C}D_{a^{+}}^{\alpha ;\rho }h(t) &=&\frac{1}{\Gamma (1-\alpha )}%
\int_{a}^{t}\left( \frac{t^{\rho }-s^{\rho }}{\rho }\right) ^{-\alpha }\frac{%
h_{\rho }^{(1)}(s)}{s^{1-\rho }}ds\text{, }t>a, \\
&=&I_{a^{+}}^{1-\alpha ;\rho }h_{\rho }^{(1)}(t).
\end{eqnarray*}%
Also, if $\alpha \in
\mathbb{N}
,$ $^{C}D_{a^{+}}^{\alpha ;\rho }h(t)=h_{\rho }^{(n)}(t).$ Particularly, $%
^{C}D_{a^{+}}^{0;\rho }h(t)=h_{\rho }^{(0)}(t)=h(t).$
\end{definition}

\begin{lemma}
$I_{a^{+}}^{\alpha ;\rho }$ is bounded on the function space $%
X_{c}^{p}(a,b). $
\end{lemma}

\begin{lemma}
\label{L2} Let $\alpha >0,$ $\beta >0,$ $h$ $\in X_{c}^{p}(a,b)$ $(1\leq
p<\infty ),$ $\rho ,c\in
\mathbb{R}
,$ $\rho \geq c.$ Then we have%
\begin{equation*}
I_{a^{+}}^{\alpha ;\rho }I_{a^{+}}^{\beta ;\rho }h(t)=I_{a^{+}}^{\alpha
+\beta ;\rho }h(t),\qquad ^{c}D_{a^{+}}^{\alpha ;\rho }\text{ }^{\rho
}I_{a^{+}}^{\alpha }h(t)=h(t).
\end{equation*}
\end{lemma}

\begin{lemma}
\label{L5} Let $t>a,$ $\alpha ,\delta \in (0,\infty ),$ and $%
I_{a^{+}}^{\alpha ;\rho },D_{a^{+}}^{\alpha ;\rho }$ and $%
^{C}D_{a^{+}}^{\alpha ;\rho }$are according on (\ref{a1}), (\ref{e2}) and (%
\ref{e3}) respectively. Then we have%
\begin{equation*}
I_{a^{+}}^{\alpha ;\rho }\left( \frac{t^{\rho }-a^{\rho }}{\rho }\right)
^{\delta -1}=\frac{\Gamma (\delta )}{\Gamma (\delta +\alpha )}\left( \frac{%
t^{\rho }-a^{\rho }}{\rho }\right) ^{\alpha +\delta -1},
\end{equation*}%
\begin{equation*}
^{C}D_{a^{+}}^{\alpha ;\rho }\left( \frac{t^{\rho }-a^{\rho }}{\rho }\right)
^{\delta -1}=\frac{\Gamma (\delta )}{\Gamma (\delta -\alpha )}\left( \frac{%
t^{\rho }-a^{\rho }}{\rho }\right) ^{\delta -\alpha -1},
\end{equation*}%
and%
\begin{equation*}
^{C}D_{a^{+}}^{\alpha ;\rho }\left( \frac{t^{\rho }-a^{\rho }}{\rho }\right)
^{k}=0,\text{ \ }\alpha \geq 0,\text{ }k=0,1,...,n-1.
\end{equation*}%
Particularly, $^{C}D_{a^{+}}^{\alpha ;\rho }\left( 1\right) =0.$
\end{lemma}

\begin{theorem}
\cite{BK} (Banach fixed point theorem) Let $(X;d)$ be a nonempty complete
metric space with $T:X\rightarrow X$ is a contraction mapping. Then map $T$
has a fixed point.
\end{theorem}

\begin{theorem}
\cite{BK} (Krasnoselskii's fixed point theorem) Let $X$ be a Banach space,
let $\Omega $ be a bounded closed convex subset of $X$ and let $T_{1},T_{2}$
be mapping from $\Omega $ into $X$ such that $T_{1}x+T_{2}y$ $\in $\ $\Omega
$ for every pair $x,y\in \Omega $. If $T_{1}$ is contraction and $T_{2}$ is
completely continuous, then the equation $T_{1}x+T_{2}x$ $=x$ has a solution
on $\Omega .$
\end{theorem}

\begin{theorem}
\cite{BK} (Leray-Schauder Nonlinear Alternative). Let $X$ be a Banach space
and $\Omega \subseteq X$ closed and convex. Assume that $K$ is a relatively
open subset of $\Omega $ with $0\in K$ and $T:\overline{K}\longrightarrow
\Omega $ is a compact and continuous mapping. Then ethier

\begin{enumerate}
\item $T$ has a fixed point in $\overline{K},$ or

\item there exists $x\in \partial K$ such that $x=\lambda Tx$ for some $%
\lambda \in (0,1)$, where $\partial K$ is boundary of $K.$
\end{enumerate}
\end{theorem}

\section{Existence and uniqueness theorems}

In this partition, we present the existence and uniqueness results of
fractional boundary value problem (\ref{1})-(\ref{2}) involving CK
fractional derivatives. To prove the existence of solutions to (\ref{1})-(%
\ref{2}), we need the following auxiliary Lemmas

\begin{lemma}
\label{L3} Let $\alpha ,\rho >0\ $and $y\in \mathcal{C(}J,%
\mathbb{R}
)\cap \mathcal{C}^{1}\mathcal{(}J,%
\mathbb{R}
).$ Then

\begin{enumerate}
\item the CK fractional deferential equation%
\begin{equation*}
^{c}D_{a^{+}}^{\alpha ;\rho }y(t)=0
\end{equation*}%
has a solutions
\begin{equation*}
y(t)=c_{0}+c_{1}\left( \frac{t^{\rho }-a^{\rho }}{\rho }\right) +c_{2}\left(
\frac{t^{\rho }-a^{\rho }}{\rho }\right) ^{2}+....+c_{n-1}\left( \frac{%
t^{\rho }-a^{\rho }}{\rho }\right) ^{n-1},
\end{equation*}%
where $c_{i}\in
\mathbb{R}
$, $i=0,1,2,...,n-1$ and $n=[\alpha ]+1.$

\item If $y,^{C}D_{a^{+}}^{\alpha ;\rho }y\in \mathcal{C(}J,%
\mathbb{R}
)\cap \mathcal{C}^{1}\mathcal{(}J,%
\mathbb{R}
)$. Then
\begin{equation}
I_{a^{+}}^{\alpha ;\rho }\text{ }^{C}D_{a^{+}}^{\alpha ;\rho
}y(t)=y(t)+c_{0}+c_{1}\left( \frac{t^{\rho }-a^{\rho }}{\rho }\right)
+c_{2}\left( \frac{t^{\rho }-a^{\rho }}{\rho }\right)
^{2}+....+c_{n-1}\left( \frac{t^{\rho }-a^{\rho }}{\rho }\right) ^{n-1},
\label{u}
\end{equation}
where $c_{i}\in
\mathbb{R}
$, $i=0,1,2,...,n-1$ and $n=[\alpha ]+1.$
\end{enumerate}
\end{lemma}

\begin{proof}
The first part, follows immediately from the fact%
\begin{equation*}
^{C}D_{a^{+}}^{\alpha ;\rho }\left( \frac{t^{\rho }-a^{\rho }}{\rho }\right)
^{k}=0,\text{ \ }k=0,1,2,...,n-1.
\end{equation*}

To prove the second part, we apply the operator $^{C}D_{a^{+}}^{\alpha ;\rho
}$ to $I_{a^{+}}^{\alpha ;\rho }$ $^{C}D_{a^{+}}^{\alpha ;\rho }y(t)-y(t),$
and use Lemma \ref{L2}, it follows that%
\begin{equation*}
^{C}D_{a^{+}}^{\alpha ;\rho }\left[ I_{a^{+}}^{\alpha ;\rho }\text{ }%
^{C}D_{a^{+}}^{\alpha ;\rho }y(t)-y(t)\right] =\text{ }^{C}D_{a^{+}}^{\alpha
;\rho }I_{a^{+}}^{\alpha ;\rho }\text{ }^{C}D_{a^{+}}^{\alpha ;\rho
}y(t)-^{C}D_{a^{+}}^{\alpha ;\rho }y(t)=0.
\end{equation*}

From the first part, we deduce there exist $c_{i}\in
\mathbb{R}
$ ($i=0,1,2,...,n-1$) such that%
\begin{equation*}
I_{a^{+}}^{\alpha ;\rho }\text{ }^{C}D_{a^{+}}^{\alpha ;\rho
}y(t)-y(t)=c_{0}+c_{1}\left( \frac{t^{\rho }-a^{\rho }}{\rho }\right)
+c_{2}\left( \frac{t^{\rho }-a^{\rho }}{\rho }\right)
^{2}+....+c_{n-1}\left( \frac{t^{\rho }-a^{\rho }}{\rho }\right) ^{n-1},
\end{equation*}%
which implies the law of composition (\ref{u}). The proof is completed.
\end{proof}

\begin{lemma}
\label{L1} Let $0<\alpha <1,$ $\rho >0$ and $g\in \mathcal{C}(J,%
\mathbb{R}
).$ Then the linear anti-preiodic boundary value problem
\begin{equation}
^{c}D_{a^{+}}^{\alpha ;\rho }y(t)=g(t),\text{ \ }t\in J,  \label{eq1}
\end{equation}%
\begin{equation}
y(a)+y(b)=0,\qquad \qquad  \label{eq2}
\end{equation}%
has a unique solution defined by
\begin{equation}
y(t)=-\frac{1}{2}\frac{\rho ^{1-\alpha }}{\Gamma (\alpha )}%
\int_{a}^{b}s^{\rho -1}(b^{\rho }-s^{\rho })^{\alpha -1}g(s)ds+\frac{\rho
^{1-\alpha }}{\Gamma (\alpha )}\int_{a}^{t}s^{\rho -1}(t^{\rho }-s^{\rho
})^{\alpha -1}g(s)ds.  \label{7}
\end{equation}
\end{lemma}

\begin{proof}
Applying the Katugampola fractional integral of order $\alpha $ to both
sides of equation in (\ref{eq1}), and using Lemma \ref{L3}, we get%
\begin{equation}
y(t)=c_{0}+\frac{\rho ^{1-\alpha }}{\Gamma (\alpha )}\int_{a}^{t}s^{\rho
-1}(t^{\rho }-s^{\rho })^{\alpha -1}g(s)ds\text{,}  \label{8}
\end{equation}%
where $c_{0}\in
\mathbb{R}
$. Using (\ref{8}) in the boundary conditions of (\ref{eq2}), we get
\begin{equation*}
c_{0}=-\frac{1}{2}\frac{\rho ^{1-\alpha }}{\Gamma (\alpha )}%
\int_{a}^{b}s^{\rho -1}(b^{\rho }-s^{\rho })^{\alpha -1}g(s)ds.
\end{equation*}%
which, on substituting in (\ref{8}), yields the solution (\ref{7}). The
converse follows by direct calculations. The proof is completed.
\end{proof}

\begin{lemma}
Assume that (H$_{1}$) holds. A function $y(t)$ solves the problem (\ref{1})-(%
\ref{2}) if and only if it is a fixed-point of the operator $T:\mathcal{C}(J,%
\mathbb{R}
)\rightarrow \mathcal{C}(J,%
\mathbb{R}
)$ defined by%
\begin{eqnarray}
Ty(t) &=&-\frac{1}{2}\frac{\rho ^{1-\alpha }}{\Gamma (\alpha )}%
\int_{a}^{b}s^{\rho -1}(b^{\rho }-s^{\rho })^{\alpha -1}f(s,y(s))ds  \notag
\\
&&+\frac{\rho ^{1-\alpha }}{\Gamma (\alpha )}\int_{a}^{t}s^{\rho -1}(t^{\rho
}-s^{\rho })^{\alpha -1}f(s,y(s))ds.  \label{r}
\end{eqnarray}%
Our first result is based on the Banach fixed point theorem to obtain the
existence of a unique solution of problem (\ref{1})-(\ref{2}).

\begin{theorem}
\label{Th1} Assume that $f:J\times
\mathbb{R}
\rightarrow
\mathbb{R}
$ be a continuous function satisfying the Lipschitz condition:
\end{theorem}
\end{lemma}

\begin{description}
\item[(H$_{1}$)] There exists a constant $L_{f}>0$ such that :
\begin{equation*}
\left\vert f(t,x)-f(t,y)\right\vert \leq L\left\vert x-y\right\vert ,\text{
\ }\forall t\in J,\text{ }x,y\in
\mathbb{R}
.
\end{equation*}%
If $L\mathcal{N}<1,$ where $\mathcal{N}:=\frac{3}{2}\frac{\rho ^{-\alpha }}{%
\Gamma (\alpha +1)}(b^{\rho }-a^{\rho })^{\alpha },$ then the boundary value
problem (\ref{1})-(\ref{2}) has a unique solution on $J$.
\end{description}

\begin{proof}
Now, we first show that the operator $T:\mathcal{C}(J,%
\mathbb{R}
)\rightarrow \mathcal{C}(J,%
\mathbb{R}
)$ defined by (\ref{r}) is well-defined, i.e., we show that $T\mathcal{B}%
_{r}\subseteq \mathcal{B}_{r}\ $where
\begin{equation}
\mathcal{B}_{r}=\{y\in \mathcal{C}(J,%
\mathbb{R}
),\left\Vert y\right\Vert \leq r\},  \label{ew}
\end{equation}
with choose $r\geq \frac{\mu \mathcal{N}}{1-L\mathcal{N}},$ and $\sup_{t\in
J}\left\vert f(t,0)\right\vert =\mu <\infty .$ For any $y\in \mathcal{B}_{r}$%
, we obtain by our hypotheses that%
\begin{eqnarray*}
\left\vert Ty(t)\right\vert &\leq &\underset{t\in J}{\sup }\left\vert
Ty(t)\right\vert \\
&\leq &\underset{t\in J}{\sup }\left\{ \frac{1}{2}\frac{\rho ^{1-\alpha }}{%
\Gamma (\alpha )}\int_{a}^{b}s^{\rho -1}(b^{\rho }-s^{\rho })^{\alpha
-1}\left\vert f(s,y(s))-f(t,0)\right\vert +\left\vert f(t,0)\right\vert
ds\right. \\
&&\left. +\frac{\rho ^{1-\alpha }}{\Gamma (\alpha )}\int_{a}^{t}s^{\rho
-1}(t^{\rho }-s^{\rho })^{\alpha -1}\left\vert f(s,y(s))-f(t,0)\right\vert
+\left\vert f(t,0)\right\vert ds\right\} \\
&\leq &\frac{3}{2}\frac{\rho ^{-\alpha }}{\Gamma (\alpha +1)}(b^{\rho
}-a^{\rho })^{\alpha }\left( Lr+\mu \right) \\
&<&\left( Lr+\mu \right) \mathcal{N} \\
&\leq &r.
\end{eqnarray*}

which implies that $Ty\in \mathcal{B}_{r}.$Moreover, by (\ref{r}), and
lammas \ref{L2}, \ref{L5}, we obtain
\begin{equation*}
\text{ }^{C}D_{a^{+}}^{\alpha ;\rho }Ty(t)=\text{ }^{C}D_{a^{+}}^{\alpha
;\rho }\text{ }I_{a^{+}}^{\alpha ;\rho }f(t,y(t))=f(t,y(t)).
\end{equation*}

Since $f(t,y(t))$ is continuous on $J,$ Then $^{c}D_{a^{+}}^{\alpha ;\rho
}Ty(t)$ is continuous on $J,$ that is $T\mathcal{B}_{r}\subseteq \mathcal{B}%
_{r}.$

Next, we apply the Banach fixed point theorem to prove that $T$ has a fixed
point. Indeed, it enough to show that $T$ is contraction map$.$ Let $%
y_{1},y_{2}\in \mathcal{C}(J,%
\mathbb{R}
)$ and for $t\in J,$ we have
\begin{eqnarray*}
\left\vert Ty_{1}(t)-Ty_{2}(t)\right\vert &\leq &\frac{1}{2}\frac{\rho
^{1-\alpha }}{\Gamma (\alpha )}\int_{a}^{b}s^{\rho -1}(b^{\rho }-s^{\rho
})^{\alpha -1}\left\vert f(s,y_{1}(s))-f(s,y_{2}(s))\right\vert ds \\
&&+\frac{\rho ^{1-\alpha }}{\Gamma (\alpha )}\int_{a}^{t}s^{\rho -1}(t^{\rho
}-s^{\rho })^{\alpha -1}\left\vert f(s,y_{1}(s))-f(s,y_{2}(s))\right\vert ds
\\
&\leq &\frac{1}{2}\frac{\rho ^{1-\alpha }}{\Gamma (\alpha )}%
\int_{a}^{b}s^{\rho -1}(b^{\rho }-s^{\rho })^{\alpha -1}L\left\Vert
y_{1}-y_{2}\right\Vert ds \\
&&+\frac{\rho ^{1-\alpha }}{\Gamma (\alpha )}\int_{a}^{t}s^{\rho -1}(t^{\rho
}-s^{\rho })^{\alpha -1}L\left\Vert y_{1}-y_{2}\right\Vert ds \\
&\leq &\frac{3}{2}\frac{\rho ^{-\alpha }L}{\Gamma (\alpha +1)}(b^{\rho
}-s^{\rho })^{\alpha }\left\Vert y_{1}-y_{2}\right\Vert \\
&=&L\mathcal{N}\left\Vert y_{1}-y_{2}\right\Vert ,
\end{eqnarray*}

which gives $\left\Vert Ty_{1}-Ty_{2}\right\Vert \leq L\mathcal{N}\left\Vert
y_{1}-y_{2}\right\Vert $. The inequality $L\mathcal{N}<1$ shows that $T$ is
contraction mapping. As a consequence of Banach fixed point theorem. Then
the problem (\ref{1})-(\ref{2}) has a unique solution. This complete the
proof.
\end{proof}

Next, we prove an existence result for the problem (\ref{1})-(\ref{2}) by
using Leray--Schauder nonlinear alternative.

\begin{theorem}
\label{Th2} Assume that $f:J\times
\mathbb{R}
\rightarrow
\mathbb{R}
$ is continuous on $J.$ In addition, we assume that:
\end{theorem}

\begin{description}
\item[(H$_{2}$)] There exist two functions $\psi :%
\mathbb{R}
^{+}\rightarrow
\mathbb{R}
^{+}$ be a nondecreasing continuous$,$ and $\eta :J\rightarrow
\mathbb{R}
^{+}$ is a continuous such that
\begin{equation*}
\left\vert f(t,y)\right\vert \leq \eta (t)\psi (\left\Vert y\right\Vert ),%
\text{ \ }\forall t\in J,\text{ }y\in
\mathbb{R}
.
\end{equation*}

\item[(H$_{3}$)] There exists a constant $M>0$ such that
\begin{equation*}
\frac{\mathcal{N}\left\Vert \eta \right\Vert \psi \left( M\right) }{M}<1,
\end{equation*}%
where $\mathcal{N}\ $is defined as in Theorem \ref{Th1}$.$ Then the boundary
value problem (\ref{1})-(\ref{2}) has at least one solution on $J$.
\end{description}

\begin{proof}
Firstly, we will prove that the operator $T$ defined by (\ref{r}), maps
bounded sets into bounded sets in $\mathcal{C}(J,%
\mathbb{R}
)$. For a positive number $r$, let $\mathcal{B}_{r}$ be a bounded ball in $%
\mathcal{C}(J,%
\mathbb{R}
)$ defined by (\ref{ew}). Then, for $t\in J$, we have%
\begin{eqnarray*}
\left\vert Ty(t)\right\vert  &\leq &\frac{1}{2}\frac{\rho ^{1-\alpha }}{%
\Gamma (\alpha )}\int_{a}^{b}s^{\rho -1}(b^{\rho }-s^{\rho })^{\alpha
-1}\left\vert f(s,y(s))\right\vert ds \\
&&+\frac{\rho ^{1-\alpha }}{\Gamma (\alpha )}\int_{a}^{t}s^{\rho -1}(t^{\rho
}-s^{\rho })^{\alpha -1}\left\vert f(s,y(s))\right\vert ds \\
&\leq &\frac{1}{2}\frac{\rho ^{1-\alpha }}{\Gamma (\alpha )}%
\int_{a}^{b}s^{\rho -1}(b^{\rho }-s^{\rho })^{\alpha -1}\eta (s)\psi \left(
\left\Vert y\right\Vert \right) ds \\
&&+\frac{\rho ^{1-\alpha }}{\Gamma (\alpha )}\int_{a}^{t}s^{\rho -1}(t^{\rho
}-s^{\rho })^{\alpha -1}\eta (s)\psi \left( \left\Vert y\right\Vert \right)
ds \\
&\leq &\frac{3}{2}\frac{\rho ^{-\alpha }}{\Gamma (\alpha +1)}(b^{\rho
}-a^{\rho })^{\alpha }\left\Vert \eta \right\Vert \psi \left( r\right)  \\
&=&\mathcal{N}\left\Vert \eta \right\Vert \psi \left( r\right) .
\end{eqnarray*}

In view of (H$_{3}$), we obtain $\left\Vert Ty\right\Vert \leq r,$ i.e. $%
\left( T\mathcal{B}_{r}\right) $ is uniformly bounded.

Next, we prove that $T$ maps bounded sets into equicontinuous sets of $%
\mathcal{C}(J,%
\mathbb{R}
)$, i.e. $\left( T\mathcal{B}_{r}\right) $ is equicontinuous. Let $%
t_{1},t_{2}\in J,$ with $t_{1}<t_{2}$ and for any $y$ $\in \mathcal{B}_{r}$,
then we have%
\begin{eqnarray*}
\left\vert Ty(t_{1})-Ty(t_{2})\right\vert &\leq &\frac{\rho ^{1-\alpha }}{%
\Gamma (\alpha )}\int_{a}^{t_{1}}s^{\rho -1}\left[ (t_{1}^{\rho }-s^{\rho
})^{\alpha -1}-(t_{2}^{\rho }-s^{\rho })^{\alpha -1}\right] \left\vert
f(s,y(s)\right\vert )ds \\
&&+\frac{\rho ^{1-\alpha }}{\Gamma (\alpha )}\int_{t_{1}}^{t_{2}}s^{\rho
-1}(t_{2}^{\rho }-s^{\rho })^{\alpha -1}\left\vert f(s,y(s))\right\vert ds \\
&\leq &\frac{\rho ^{1-\alpha }}{\Gamma (\alpha )}\int_{a}^{t_{1}}s^{\rho -1}%
\left[ (t_{1}^{\rho }-s^{\rho })^{\alpha -1}-(t_{2}^{\rho }-s^{\rho
})^{\alpha -1}\right] \eta (s)\psi \left( \left\Vert y\right\Vert \right) ds
\\
&&+\frac{\rho ^{1-\alpha }}{\Gamma (\alpha )}\int_{t_{1}}^{t_{2}}s^{\rho
-1}(t_{2}^{\rho }-s^{\rho })^{\alpha -1}\eta (s)\psi \left( \left\Vert
y\right\Vert \right) ds \\
&\leq &\left\Vert \eta \right\Vert \psi \left( r\right) \frac{\rho
^{1-\alpha }}{\Gamma (\alpha )}\int_{a}^{t_{1}}s^{\rho -1}\left[
(t_{1}^{\rho }-s^{\rho })^{\alpha -1}-(t_{2}^{\rho }-s^{\rho })^{\alpha -1}%
\right] ds \\
&&+\left\Vert \eta \right\Vert \psi \left( r\right) \frac{\rho ^{1-\alpha }}{%
\Gamma (\alpha )}\int_{t_{1}}^{t_{2}}s^{\rho -1}(t_{2}^{\rho }-s^{\rho
})^{\alpha -1}ds \\
&\leq &\frac{2\rho ^{-\alpha }\left\Vert \eta \right\Vert \psi \left(
r\right) }{\Gamma (\alpha +1)}(t_{2}^{\rho }-t_{1}^{\rho })^{\alpha }.
\end{eqnarray*}

As $t_{1}\longrightarrow t_{2}$ the right-hand side of the preceding
inequality is not dependent on $y$ and goes to zero. Consequently, $T%
\mathcal{B}_{r}$ is equicontinuous i.e.%
\begin{equation*}
\left\vert Ty(t_{1})-Ty(t_{2})\right\vert \rightarrow 0\text{, \ }\forall
\text{ }\left\vert t_{2}-t_{1}\right\vert \rightarrow 0,\text{ }y\in
\mathcal{B}_{r}.
\end{equation*}%
So, the compactness of $T$ follows by Ascoli Arzela's theorem, we conclude
that $T$ is completely continuous.

Finally, we show there exists an open set $\mathcal{U}\subseteq \mathcal{C}%
(J,%
\mathbb{R}
)$ with $y\neq \lambda Ty$ for $\lambda \in (0,1)$ and $y\in \partial
\mathcal{U}$.

Let $y\in \mathcal{B}_{r}$ be any solution of
\begin{equation}
y=\lambda Ty,\text{ }\lambda \in (0,1).  \label{E3'}
\end{equation}

Then%
\begin{eqnarray}
\left\vert y(t)\right\vert  &=&\lambda \left\vert Ty(t)\right\vert   \notag
\\
&<&\frac{1}{2}\frac{\rho ^{1-\alpha }}{\Gamma (\alpha )}\int_{a}^{b}s^{\rho
-1}(b^{\rho }-s^{\rho })^{\alpha -1}\left\vert f(s,y(s))\right\vert ds
\notag \\
&&+\frac{\rho ^{1-\alpha }}{\Gamma (\alpha )}\int_{a}^{t}s^{\rho -1}(t^{\rho
}-s^{\rho })^{\alpha -1}\left\vert f(s,y(s))\right\vert ds  \notag \\
&\leq &\frac{1}{2}\frac{\rho ^{1-\alpha }}{\Gamma (\alpha )}%
\int_{a}^{b}s^{\rho -1}(b^{\rho }-s^{\rho })^{\alpha -1}\eta (s)\psi \left(
\left\Vert y\right\Vert \right) ds  \notag \\
&&+\frac{\rho ^{1-\alpha }}{\Gamma (\alpha )}\int_{a}^{t}s^{\rho -1}(t^{\rho
}-s^{\rho })^{\alpha -1}\eta (s)\psi \left( \left\Vert y\right\Vert \right)
ds  \notag \\
&\leq &\frac{3}{2}\left\Vert \eta \right\Vert \psi \left( \left\Vert
y\right\Vert \right) \frac{\rho ^{-\alpha }}{\Gamma (\alpha +1)}(b^{\rho
}-a^{\rho })^{\alpha },  \label{E1'}
\end{eqnarray}%
which leads to
\begin{equation*}
\frac{\left\Vert y\right\Vert }{\mathcal{N}\left\Vert \eta \right\Vert \psi
\left( \left\Vert y\right\Vert \right) }<1.
\end{equation*}%
In view of (H$_{2}$), there exists $M$ such that $\left\Vert y\right\Vert
\neq M$. This means that, any solution $y$ of equation (\ref{E3'}) satisfies
$\left\Vert y\right\Vert \neq M,$ let
\begin{equation*}
\mathcal{U=\{}z\in \mathbb{K};\left\Vert y\right\Vert <M\}.
\end{equation*}%
Thus, the Leray--Schauder nonlinear alternative guarantees that (\ref{E3'})
has a fixed point in $\partial \mathcal{U}$, which is a solution of the
boundary value problem (\ref{1})-(\ref{2}), The proof is over.
\end{proof}

We will study the next existence result by using Krasnoselskii$^{^{\prime }}$%
s fixed point theorem. To this end, we change hypothesis (H$_{2}$) into the
following one:

\begin{description}
\item[(H$_{4}$)] There exists a function $q(t)\in \mathcal{C}(J,%
\mathbb{R}
)$ suth that
\begin{equation*}
\left\vert f(t,y)\right\vert \leq q(t),\text{ \ }\forall t\in J,\text{ }y\in
\mathbb{R}
.
\end{equation*}
\end{description}

\begin{theorem}
\label{Th3} Assume that $(H_{4})$ holds. In addition, we assume that:
\end{theorem}

\begin{description}
\item[$($H$_{5})$ ] There exists a function $\delta (t)\in \mathcal{C}(J,%
\mathbb{R}
)$ suth that
\begin{equation*}
\left\vert f(t,x)-f(t,y)\right\vert \leq \delta (t)\left\vert x-y\right\vert
,\text{ }\forall t\in J,\text{ }x,y\in
\mathbb{R}
.
\end{equation*}%
If
\begin{equation}
\Lambda :=\frac{1}{2}\frac{\rho ^{-\alpha }\left\Vert \delta \right\Vert }{%
\Gamma (\alpha +1)}(b^{\rho }-a^{\rho })^{\alpha }<1,  \label{w}
\end{equation}%
then the boundary value problem (\ref{1})-(\ref{2}) has at least one
solution on $J$.
\end{description}

\begin{proof}
Consider the operator $T:\mathcal{C}(J,%
\mathbb{R}
)\longrightarrow \mathcal{C}(J,%
\mathbb{R}
)$ defined by (\ref{r}). We define $\mathcal{B}_{r_{0}}:=\{y\in \mathcal{C}%
(J,%
\mathbb{R}
):\left\Vert y\right\Vert \leq r_{0}\}$, $\sup_{t\in J}\left\vert
q(t)\right\vert =\left\Vert q\right\Vert ,$ and select a suitable constant $%
r_{0}$ such that
\begin{equation*}
r_{0}=\mathcal{N}\left\Vert q\right\Vert +1,
\end{equation*}%
where $\mathcal{N}\ $is defined as in Theorem \ref{Th1}. Furthermore, we
need to analyze the operator $T$ into sum two operators $T_{1}+T_{2}$, as
follows
\begin{equation*}
T_{1}y(t)=-\frac{1}{2}\frac{\rho ^{1-\alpha }}{\Gamma (\alpha )}%
\int_{a}^{b}s^{\rho -1}(b^{\rho }-s^{\rho })^{\alpha -1}f(s,y(s))ds,
\end{equation*}%
and%
\begin{equation*}
T_{2}y(t)=\frac{\rho ^{1-\alpha }}{\Gamma (\alpha )}\int_{a}^{t}s^{\rho
-1}(t^{\rho }-s^{\rho })^{\alpha -1}f(s,y(s))ds.
\end{equation*}%
Taking into account that $T_{1}$ and $T_{2}$ are defined on $\mathcal{B}%
_{r_{0}}.$ The proof will be given in several steps.

\textbf{Step 1}: $T_{1}y_{1}+T_{2}y_{2}$ $\in \mathcal{B}_{r_{0}}$ for every
$y_{1},y_{2}\in \mathcal{B}_{r_{0}}$.

For $y_{1},y_{2}\in \mathcal{B}_{r_{0}}$, we have%
\begin{eqnarray*}
\left\vert T_{1}y_{1}(t)+T_{2}y_{2}(t)\right\vert  &\leq &\left\vert
T_{1}y_{1}(t)\right\vert +\left\vert T_{2}y(t)\right\vert  \\
&\leq &\frac{1}{2}\frac{\rho ^{1-\alpha }}{\Gamma (\alpha )}%
\int_{a}^{b}s^{\rho -1}(b^{\rho }-s^{\rho })^{\alpha -1}\left\vert
f(s,y_{1}(s))\right\vert ds \\
&&+\frac{\rho ^{1-\alpha }}{\Gamma (\alpha )}\int_{a}^{t}s^{\rho -1}(t^{\rho
}-s^{\rho })^{\alpha -1}\left\vert f(s,y_{2}(s))\right\vert ds \\
&\leq &\frac{3}{2}\frac{\rho ^{-\alpha }}{\Gamma (\alpha +1)}(b^{\rho
}-a^{\rho })^{\alpha }\left\Vert q\right\Vert ,
\end{eqnarray*}%
which gives
\begin{equation}
\left\Vert T_{1}y_{1}+T_{2}y_{2}\right\Vert \leq \mathcal{N}\left\Vert
q\right\Vert \leq r_{0}.  \label{q1}
\end{equation}

This proves that $T_{1}y_{1}+T_{2}y_{2}$ $\in \mathcal{B}_{r_{0}}$ for every
$y_{1},y_{2}\in \mathcal{B}_{r_{0}}.$

\textbf{Step 2}: $T_{1}$ is a contration mapping on $\mathcal{B}_{r_{0}}$.

In view of hypothesis (H$_{5}$), then for each $y_{1},y_{2}\in \mathcal{B}%
_{r_{0}}$ and $t\in $ $J$; we have%
\begin{eqnarray*}
\left\vert T_{1}y_{1}(t)-T_{1}y_{2}(t)\right\vert &\leq &\frac{1}{2}\frac{%
\rho ^{1-\alpha }}{\Gamma (\alpha )}\int_{a}^{b}s^{\rho -1}(b^{\rho
}-s^{\rho })^{\alpha -1}\left\vert f(s,y_{1}(s))-f(s,y_{2}(s))\right\vert ds
\\
&\leq &\frac{1}{2}\frac{\rho ^{1-\alpha }}{\Gamma (\alpha )}%
\int_{a}^{b}s^{\rho -1}(b^{\rho }-s^{\rho })^{\alpha -1}\left\Vert \delta
\right\Vert \left\Vert y_{1}-y_{2}\right\Vert ds \\
&\leq &\frac{1}{2}\frac{\rho ^{-\alpha }\left\Vert \delta \right\Vert }{%
\Gamma (\alpha +1)}(b^{\rho }-a^{\rho })^{\alpha }\left\Vert
y_{1}-y_{2}\right\Vert ,
\end{eqnarray*}%
which implies%
\begin{equation*}
\left\Vert T_{1}y_{1}-T_{1}y_{2}\right\Vert \leq \Lambda \left\Vert
y_{1}-y_{2}\right\Vert ,
\end{equation*}%
It follows from the inequality (\ref{w}) that $T_{1}$ is contraction mapping.

\textbf{Step 3}: The operator $T_{2}$ is completely continuous on $\mathcal{B%
}_{r_{0}}$.

First, from the continuity of $f$, we conclude that the operator $T_{2}$ is
continuous.

Next, It is easy to verify that%
\begin{equation*}
\left\Vert T_{2}y\right\Vert \leq \frac{\rho ^{-\alpha }\left\Vert
q\right\Vert }{\Gamma (\alpha +1)}(b^{\rho }-a^{\rho })^{\alpha }=\frac{2}{3}%
\mathcal{N}\left\Vert q\right\Vert =\frac{2}{3}r_{0}-1<r_{0}.
\end{equation*}

This proves that $T_{2}$ is uniformly bounded on $\mathcal{B}_{r_{0}}$.

Finally, we show that $T_{2}$ is equicontinuous on $\mathcal{B}_{r_{0}}$.
Let us set $sup_{(t,y)\in J\times \mathcal{B}_{r_{0}}}$ $\left\vert
f(t,y\right\vert =f_{0}<\infty ,$\ and let $y\in \mathcal{B}_{r_{0}}$ and $%
t_{1},t_{2}\in $ $J$ with $t_{1}<t_{2}$. Then we have%
\begin{eqnarray*}
\left\vert T_{2}y(t_{1})-T_{2}y(t_{2})\right\vert &\leq &\frac{\rho
^{1-\alpha }}{\Gamma (\alpha )}\int_{a}^{t_{1}}s^{\rho -1}\left[
(t_{1}^{\rho }-s^{\rho })^{\alpha -1}-(t_{2}^{\rho }-s^{\rho })^{\alpha -1}%
\right] \left\vert f(s,y(s)\right\vert ds \\
&&+\frac{\rho ^{1-\alpha }}{\Gamma (\alpha )}\int_{t_{1}}^{t_{2}}s^{\rho
-1}(t_{2}^{\rho }-s^{\rho })^{\alpha -1}\left\vert f(s,y(s))\right\vert ds \\
&\leq &f_{0}\frac{\rho ^{1-\alpha }}{\Gamma (\alpha )}\int_{a}^{t_{1}}s^{%
\rho -1}\left[ (t_{1}^{\rho }-s^{\rho })^{\alpha -1}-(t_{2}^{\rho }-s^{\rho
})^{\alpha -1}\right] ds \\
&&+f_{0}\frac{\rho ^{1-\alpha }}{\Gamma (\alpha )}\int_{t_{1}}^{t_{2}}s^{%
\rho -1}(t_{2}^{\rho }-s^{\rho })^{\alpha -1}ds \\
&\leq &\frac{\rho ^{-\alpha }f_{0}}{\Gamma (\alpha +1)}(t_{2}^{\rho
}-t_{1}^{\rho })^{\alpha }.
\end{eqnarray*}

As $t_{1}\longrightarrow t_{2}$ the right-hand side of the last inequality
is not dependent on $y$ and goes to zero. Consequently,%
\begin{equation*}
\left\vert T_{2}y(t_{1})-T_{2}y(t_{2})\right\vert \rightarrow 0\text{, \ }%
\forall \text{ }\left\vert t_{2}-t_{1}\right\vert \rightarrow 0,\text{ }y\in
\mathcal{B}_{r_{0}}.
\end{equation*}%
This proves that $T_{2}$ is equicontinuous on $\mathcal{B}_{r_{0}}$. In view
of Arzela-Ascoli Theorem, it follows that $T_{2}$ is relatively compact on $%
\mathcal{B}_{r_{0}}$. Thus, all the assumptions of Krasnosel'skii fixed
point theorem are satisfied. Therefore, we conclude that the boundary value
problem (\ref{1})-(\ref{2}) has at least one solution on $J.$
\end{proof}

\section{Examples}

Will be provided in the revised submission.

\end{document}